\documentclass[11pt]{amsproc}
\usepackage{pifont}
\usepackage[mathcal]{euscript}
\usepackage{amsthm}
\usepackage{hyperref}
\usepackage{cleveref}
\usepackage{amsfonts}
\usepackage{amsmath,amssymb}
\usepackage{indentfirst,latexsym, bm,amsthm,graphicx,colortbl}
\usepackage{fancyhdr}
\usepackage{times}
\usepackage{amsmath,amsfonts,amssymb,amsthm}
\usepackage{arydshln}      
\usepackage{extarrows}
\usepackage{tikz}          
\usepackage{enumitem}      
\usepackage{etoolbox}
\newcommand*{\circled}[1]{%
\tikz[baseline=(char.base)]{%
    \node[shape=circle, draw, inner sep=0.5pt, minimum size=1em] (char) {\small #1};%
}%
}
\robustify{\circled}
\newtheorem{theorem}{Theorem}[section]

\newtheorem{notation}[theorem]{Notation}
\newtheorem{prop}[theorem]{Proposition}
\newtheorem{theo}[theorem]{Theorem}
\theoremstyle{definition}
\newtheorem{defi}[theorem]{Definition}
\newtheorem{example}[theorem]{Example}

\theoremstyle{remark}
\newtheorem{remark}[theorem]{Remark}
\makeatletter

\newcommand{\Rmnum}[1]{\expandafter\@slowromancap\romannumeral #1@}
\makeatother

\numberwithin{equation}{section}

\title{On the length of a class of maximal commutative subalgebras}
\author{Chengjie Wang}

\begin{document}
\begin{abstract}
A maximal commutative subalgebra is a substructure in algebra with the greatest commutative property. By studying the lengths of maximal commutative subalgebras, one can more clearly characterize the structure of commutative subalgebras in the full matrix algebra $\mathrm{M}_n({\mathbb{F}})$. Inspired by \cite[Proposition~4.12]{markova2013}, this paper identifies a class of maximal commutative subalgebras $\mathcal{B}_{k,m,l}$ and computes their lengths. Finally, we present two concrete examples to show that it is not a straightforward generalization.
\medskip

\raggedright\text{Keywords:} {the full matrix algebra; maximal commutative subalgebras; length of an algebra; generating systems; generating systems of an algebra}
\end{abstract}
\maketitle

\section{Introduction}
\label{S1}
Maximal commutative subalgebras of the full matrix algebra $\mathrm{M}_n({\mathbb{F}})$ are fundamentally important in many areas of mathematics. They play a crucial role in studying algebraic structures, analyzing matrix algebras, and advancing related research fields. 

Due to differences in research objects and conditions, the construction methods for maximal commutative subalgebras are also diverse. In 1991, Laffey and Lazarus \cite{ls1991} investigated the commutative subalgebra $\mathcal{A}$ generated by a pair of commuting matrices $\mathcal{S}=\{A, B\}$. It was established that for a matrix $A$ similar to $J_{k}\oplus \cdots \oplus J_k$ (where all Jordan blocks are identical in size) over an algebraically closed field $\mathbb{F}$, the subalgebra $\mathcal{A}$ is maximal commutative if and only if ${\rm dim}_{\mathbb{F}}\mathcal{A}=n$. In 1993, Brown and Call \cite{bc1993} conducted a thorough investigation of maximal commutative subalgebras in the full matrix algebra $\mathrm{M}_n({\mathbb{F}})$. They proposed the $(B, N)$-construction and established that this approach could generate certain classical commutative subalgebras of dimension less than $n$, including Courter's algebra. In 1994, Brown \cite{brown1994} proposed two additional constructions: the $C_1$ and $C_2$-constructions, which yield distinct classes of maximal commutative subalgebras that are pairwise non-inclusive. In 1997, Brown \cite{brown1997} proved that for the full matrix algebra $\mathrm{M}_n({\mathbb{F}})$ of low order (i.e., $n\le 5$) over an algebraically closed field $\mathbb{F}$, all maximal commutative subalgebras $\mathcal{A}\in L$ (except for one isomorphism class) are either $C_1$ or $C_2$-constructions. Here, the set $L$ is defined as 
$$L=\{(\mathcal{B}, J,\mathbb{F}) \vert~\mathcal{B}~\text{is a local algebra with nilpotency index}~N(J)\ge 2\}.$$ 
In the same year, Song \cite{song1997} constructed a new class of maximal commutative subalgebras 
$$S=\mathbb{F}[\delta_1, \cdots, \delta_8, E_{1,1}, E_{1,2}, E_{2,1}, E_{2,2}],$$ 
proving that it is non-isomorphic to the known Courter's algebra. Consequently, it was established that for maximal commutative subalgebras of $14\times14$ matrices whose dimension is $13$ and index of Jacobson radical is $3$, there are at least two isomorphism classes. Furthermore, the study revealed that the $(B, N)$-construction depends on the field $\mathbb{F}$. In 2002, Song \cite{song2002} further investigated maximal commutative subalgebras with nilpotency index $3$ and specific socle forms, explicitly constructing a new class of maximal commutative subalgebras isomorphic to them. Subsequently, in 2003, Song \cite{song2003} introduced the $C_{2}^{2}$-construction for building maximal commutative subalgebras of $\mathrm{M}_n({\mathbb{F}})$, advancing the development of this research field. 

In this paper, we construct a class of maximal commutative subalgebras and determine their lengths, establishing the following result (see Theorems \ref{dinglijida} and \ref{dinglijida1}):
{
\renewcommand{\thetheorem}{\Alph{theorem}}
\begin{theo}\label{d}
Let $m, l\in \mathbb{N}$, $k\in \mathbb{Z}^+$, $l>m+k+1$, and $l+k+1\le n$. Consider the subalgebra $\mathcal{B}_{k,m,l}\subseteq \mathrm{M}_n(\mathbb{F})$ generated by the matrices
$$\mathbb{E}_n, B_1=E_{m,m+1}+E_{m+1,m+2}+\cdots+E_{m+k,m+k+1}, $$
$$B_2=E_{l,l+1}+E_{l+1,l+2}+\cdots+E_{l+k,l+k+1}, E_{i,j},$$
where $1\le i\le m$ or $i=l$, and $m+k+1\le j\le l-1$ or $l+k+1\le j\le n$. Then $\mathcal{B}_{k,m,l}$ is a maximal commutative subalgebra in $\mathrm{M}_n(\mathbb{F})$ of length $\ell\left(\mathcal{B}_{k,m,l}\right)=k+1$.
\end{theo}
}
Theorem \ref{d}, inspired by \cite[Proposition~4.12]{markova2013}, is illustrated with two concrete examples demonstrating that it does not generalize \cite[Proposition~4.12]{markova2013}.

The paper is organized as follows. Section \ref{S2} contains some definitions and auxiliary results;
Section \ref{S3} is devoted to constructing a class of maximal commutative subalgebras $\mathcal{B}_{k,m,l}$ and computing the length of this algebra.

\section{Preliminaries}
\label{S2}
Let $\mathcal{A}$ be a finite-dimensional associative algebra with identity over an arbitrary field $\mathbb{F}$. Since $\mathcal{A}$ is finite-dimensional over $\mathbb{F}$, it is evidently finite generated. Let $\mathcal{S}=\{a_1, a_2, \cdots, a_k\}$ be a finite generating set of this algebra. 

Firstly, it is necessary to define the length of a word of $\mathcal{S}$.

\begin{defi}
Words in $\mathcal{S}$ are products of a finite number of elements from the alphabet $\mathcal{S}$. The length of a word equals the number of factors in the corresponding product. 
\end{defi}

\begin{remark}
  The length of the word $a_{i_1}\cdots a_{i_t}$, where $a_{i_j}\in\mathcal{S}$, is equal to $t$. The unity $1$ of the algebra is considered as a word of length $0$ over $\mathcal{S}$ and also call it the empty word.  
\end{remark}

\begin{defi}
For $i\geq 0$, let $\mathcal{S}^i$ denote the set of all words in the alphabet $\mathcal{S}$ of length less than or equal to $i$. Let $\mathcal{S}^i\backslash\mathcal{S}^{i-1}$ be the set of all words of length $i$ over $\mathcal{S}$, $i\geq 1$.
\end{defi}

\begin{defi}
The linear span of $ \mathcal{S}$ in a vector space over $\mathbb{F}$, denoted by $\langle\mathcal{S}\rangle$, is the set of all finite $\mathbb{F}$-linear combinations.
\end{defi}

The length functions associated with a generating set $\mathcal{S}$ and an algebra $\mathcal{A}$ are formally defined as follows.

\begin{defi}
Let $L_i\left(\mathcal{S}\right)=\langle\mathcal{S}^i\rangle$ denote the linear span of the words in $\mathcal{S}^i$. Note that 
$$L_0\left(\mathcal{S}\right)=\langle 1\rangle = \mathbb{F}$$
for unitary algebras. Also set $L\left(\mathcal{S}\right)$=$\bigcup\limits_{i=0}^{\infty}L_i\left(\mathcal{S}\right)$ be the linear span of all words in the alphabet $\mathcal{S}$.
\end{defi}

\begin{defi}
A set $\mathcal{S}$ is a generating system for $\mathcal{A}$ if and only if $\mathcal{A}$=$L\left(\mathcal{S}\right)$.
\end{defi}

\begin{remark}
    Since $\mathcal{A}$ is finite-dimensional over $\mathbb{F}$ if follows that the chain 
\begin{equation*}
    \mathbb{F}=L_0(\mathcal{S})\subset L_1(\mathcal{S})\subset L_2(\mathcal{S})\subset \cdots
\end{equation*}       
stabilizes. That is, there is an integer $k$ such that $L_k(\mathcal{S})= L_{k+i}(\mathcal{S})$, for all $i\ge 0$. We must have $L_k(\mathcal{S})=\mathcal{A}$ , since $\mathcal{S}$
is a generating set for $\mathcal{A}$.
\end{remark}

\begin{defi}
The length of a generating system $\mathcal{S}$ of a finite-dimensional algebra $\mathcal{A}$ is defined as the minimum non-negative integer $k$ such that $L_k\left(\mathcal{S}\right)=\mathcal{A}$. The length of $\mathcal{S}$ is denoted by 
\begin{equation*}
\ell\left(\mathcal{S}\right)=\mathop{\min}\{k\in\mathbb{Z}^+:L_k\left(\mathcal{S}\right)=\mathcal{A}\}.
\end{equation*}
\end{defi}

\begin{defi}
The length of an algebra $\mathcal{A}$ is defined as the maximum of the lengths of its generating systems $\mathcal{S}$ such that $L\left(\mathcal{S}\right)=\mathcal{A}$. The length of $\mathcal{A}$ is denoted by 
\begin{equation*}
\ell\left(\mathcal{A}\right)=\mathop{\max}\{\ell\left(\mathcal{S}\right):L\left(\mathcal{S}\right)=\mathcal{A}\}.
\end{equation*}
\end{defi}

\begin{notation}
Let $x \in \mathnormal{R}$. The integer part of $x$, denoted $\lfloor{x}\rfloor$, is defined as the largest integer not exceeding $x$. 
\end{notation}

\begin{notation}
Let $\mathrm{M}_n(\mathbb{F})$ denote the algebra of $\mathnormal{n}\times\mathnormal{n}$ matrices over an arbitrary field $\mathbb{F}$. 

$\bullet$ $E_{ij}:$ The $\left(i, j\right)$-th matrix unit, i.e., the matrix with $1$ at the $\left(i, j\right)$ position and zeros at the remaining positions. 

$\bullet$ $E_n:$ The $n\times n$ identity matrix. 
\end{notation}

Next, we will introduce some fundamental definitions regarding the maximal commutative subalgebras of the full matrix algebra  \(\mathrm{M}_n(\mathbb{F})\).

\begin{defi}
\label{dingyijida}
If $\mathcal{B}\subseteq \mathrm{M}_n(\mathbb{F})$ is a commutative subalgebra such that there exists no commutative subalgebra $\mathcal{B}^{\prime}$ in $\mathrm{M}_n(\mathbb{F})$ satisfying $\mathcal{B}\subsetneq \mathcal{B}^{\prime}$ (i.e., $\mathcal{B}$ is maximal under inclusion among all commutative subalgebras of $\mathrm{M}_n(\mathbb{F})$), then $\mathcal{B}$ is called a \textbf{maximal commutative subalgebra}.
\end{defi}

The following provides an equivalent characterization of maximal commutative subalgebras:

\begin{remark}
$\mathcal{B}$  is a maximal commutative subalgebra if and only if its centralizer in $\mathrm{M}_n(\mathbb{F})$ coincides with itself, that is,
$$C_{\mathrm{M}_n(\mathbb{F})}(\mathcal{B})=\{X\in\mathrm{M}_n(\mathbb{F})~\vert~\text{for all}~A\in \mathcal{B},~\text{we have}~XA=XA\}=\mathcal{B}.$$
\end{remark}

\begin{defi}
An associative algebra $\mathcal{A}$ is called \textbf{local} if it has a unique maximal right ideal.
\end{defi}

\begin{defi}
The \textbf{Jacobson radical} of an associative ring is the intersection of all its maximal right (left) ideals.
\end{defi}

\begin{remark}
It is well known that the set of all noninvertible elements of a local algebra coincides with its Jacobson radical. The Jacobson radical $J$ of an Artinian ring is nilpotent, i.e., there exists a number $N$ such that $J^N=(0)$, but $J^{N-1}\neq (0)$; here $J^N$ denotes the set of products of elements from $J$ of length $N$, and the number $N$ is called the nilpotency index of the ideal $J$. In particular, the Jacobson radical of a finite-dimensional algebra is nilpotent.
\end{remark}

\begin{theo}[{\cite[Corollary~3.32]{markova2013}}]
\label{dingli2}
Let $\mathbb{F}$ be an arbitrary field, and $\mathcal{A}$ be a finite-dimensional local $\mathbb{F}$-algebra with the Jacobson radical $J(\mathcal{A})$. Let also $N$ stand for the nilpotency index of $J(\mathcal{A})$. Let $\mathcal{A}=\mathbb{F}1_{\mathcal{A}} +J(\mathcal{A})$. Then $\ell\left(\mathcal{S}\right)\le N-1$.
\end{theo}

\begin{prop}[{\cite[Proposition~4.12]{markova2013}}]
\label{mingtijida}
Let $m, n\in \mathbb{N}$, $k\in \mathbb{Z}^+$, and $k+m+1\le n$. Consider the subalgebra $\mathcal{B}_{k,m}\subseteq \mathrm{M}_n(\mathbb{F})$ generated by the matrices
$$\mathbb{E}_n, B=E_{m,m+1}+E_{m+1,m+2}+\cdots+E_{m+k,m+k+1}, E_{i,j},$$
where $1\le i\le m$, $m+k+1\le j\le n$. Then $\mathcal{B}_{k,m}$ is a maximal commutative subalgebra in $\mathrm{M}_n(\mathbb{F})$ of length $\ell\left(\mathcal{B}_{k,m}\right)=k+1$.
\end{prop}

\section{Results}
\label{S3}
\subsection{The Construction of a Class of Maximal Commutative Subalgebras $\mathcal{B}_{k,m,l}$}
\ \newline

In this subsection, inspired by Proposition \ref{mingtijida}, we will construct and prove a class of maximal commutative subalgebras.

\begin{theo}
\label{dinglijida}
Let $m, l\in \mathbb{N}$, $k\in \mathbb{Z}^+$, $l>m+k+1$, and $l+k+1\le n$. Consider the subalgebra $\mathcal{B}_{k,m,l}\subseteq \mathrm{M}_n(\mathbb{F})$ generated by the matrices
$$\mathbb{E}_n, B_1=E_{m,m+1}+E_{m+1,m+2}+\cdots+E_{m+k,m+k+1}, $$
$$B_2=E_{l,l+1}+E_{l+1,l+2}+\cdots+E_{l+k,l+k+1}, E_{i,j},$$
where $1\le i\le m$ or $i=l$, and $m+k+1\le j\le l-1$ or $l+k+1\le j\le n$. Then $\mathcal{B}_{k,m,l}$ is a maximal commutative subalgebra in $\mathrm{M}_n(\mathbb{F})$.
\end{theo}

\begin{proof}
The following is the proof divided into four steps.

\textbf{Step 1}:~~The structure of the algebra $\mathcal{B}_{k,m,l}$ allows us to determine the relationships between the generator matrices. These will be presented individually below.

\begin{enumerate}[label=\circled{\arabic*}]
\item\label{1} Since $l>m+k+1$, then
$$B_1B_2=(E_{m,m+1}+E_{m+1,m+2}+\cdots+E_{m+k,m+k+1})(E_{l,l+1}+\cdots+E_{l+k,l+k+1})=0.$$
Under the condition that $l>m+k+1>m$, it follows that
$$B_2B_1=(E_{l,l+1}+\cdots+E_{l+k,l+k+1})(E_{m,m+1}+E_{m+1,m+2}+\cdots+E_{m+k,m+k+1})=0.$$

\item If $1\le i\le m$, then $i<m+1$ holds trivially. Furthermore, if $i=l$ and $l>m+k+1$, we immediately obtain $i>m+k+1$. Therefore, we obtain the matrix identity:
$$B_1E_{i,j}=(E_{m,m+1}+E_{m+1,m+2}+\cdots+E_{m+k,m+k+1})E_{i,j}=0.$$
Since either $m+k<m+k+1\le j\le l-1$ or $m+k+1<l+k+1\le j\le n$ holds, we conclude that
$$E_{i,j}B_1=E_{i,j}(E_{m,m+1}+E_{m+1,m+2}+\cdots+E_{m+k,m+k+1})=0.$$

\item If $1\le i\le m$ or $i=l$, then $i<l+1$. Thus, the following equation holds:
$$B_2E_{i,j}=(E_{l,l+1}+E_{l+1,l+2}+\cdots+E_{l+k,l+k+1})E_{i,j}=0.$$
Given that $j$ satisfies either $m+k+1\le j\le l-1$ or $l+k+1\le j\le n$, that is $j\notin \{l, l+1, \cdots, l+k\}$, it follows that
$$E_{i,j}B_2=E_{i,j}(E_{l,l+1}+E_{l+1,l+2}+\cdots+E_{l+k,l+k+1})=0.$$

\item For any matrices $E_{i,j}, E_{p,q}\in \mathcal{B}_{k,m,l}$, under the given conditions:
$$1\le p\le m~\text{or}~p=l,~\text{and}~m+k+1\le j\le l-1~\text{or}~l+k+1\le j\le n,$$
we observe that $p\neq j$ always holds. Thus, the matrix product satisfies $E_{i,j}E_{p,q}=0$.

\item\label{5} Owing to the special structure of matrices~$B_1$ and $B_2$, we obtain by induction:
\begin{gather*}
B_1^s=\sum_{h=0}^{k-s+1}E_{m+h,m+h+s},~\text{where}~s\in\{1, 2, \cdots, k+1\},\\
B_2^t=\sum_{h=0}^{k-t+1}E_{l+h,l+h+t},~\text{where}~t\in\{1, 2, \cdots, k+1\}.
\end{gather*}
However, when $s,t\ge k+2$, $B_1^s=B_2^t=0$.
\end{enumerate}

Based on the analysis in equations \ref{1} - \ref{5} above, we derive the following results:
\begin{equation}
\label{diyibu}
\left\{
\begin{aligned}
&B_1B_2=B_2B_1=0, B_1E_{i,j}=E_{i,j}B_1=B_2E_{i,j}=E_{i,j}B_2=0,  \\
&\text{For any}~E_{i,j}, E_{p,q} \in \mathcal{B}_{k,m,l},~\text{we have}~E_{i,j}E_{p,q}=0,   \\
&B_1^s=\sum_{h=0}^{k-s+1}E_{m+h,m+h+s},~\text{where}~s\in\{1, 2, \cdots, k+1\}, \\
&B_2^t=\sum_{h=0}^{k-t+1}E_{l+h,l+h+t},~\text{where}~t\in\{1, 2, \cdots, k+1\}.
\end{aligned}
\right.
\end{equation}

\textbf{Step 2}:~~The following verifies that $\mathcal{B}_{k,m,l}$ is a subalgebra of $\mathrm{M}_n(\mathbb{F})$. Let us denote the set
$$M=\{j\in\mathbb{Z}^+\vert m+k+1\le j\le l-1~or~l+k+1\le j\le n\},$$
and $W=\{i\in\mathbb{Z}^+\vert 1\le i\le m~or~i=l\}$.

For any matrices $A_1, A_2\in \mathcal{B}_{k,m,l}$, based on the computational results from Step 1 Eq. (\ref{diyibu}) and the structure of $\mathcal{B}_{k,m,l}$, we make the following assumption:
\begin{equation*}
\begin{aligned}
&A_1=\gamma_1E_n+\sum_{s_1=1}^{k+1}\alpha_{1,s_1}B_1^{s_1}+\sum_{t_1=1}^{k+1}\lambda_{1,t_1}B_2^{t_1}+\sum_{i\in W}\sum_{j\in M}\mu_{1,i,j}E_{i,j},\\
&A_2=\gamma_2E_n+\sum_{s_2=1}^{k+1}\alpha_{2,s_2}B_1^{s_2}+\sum_{t_2=1}^{k+1}\lambda_{2,t_2}B_2^{t_2}+\sum_{i\in W}\sum_{j\in M}\mu_{2,i,j}E_{i,j},
\end{aligned}
\end{equation*}
where $\gamma_1, \gamma_2, \alpha_{1,s_1}, \alpha_{2,s_2}, \lambda_{1,t_1}, \lambda_{2,t_2}, \mu_{1,i,j}, \mu_{2,i,j} \in \mathbb{F}$.

Thus,
\begin{equation}
\label{dierbu}
\begin{aligned}
&\left(A_1-\gamma_1E_n\right)\left(A_2-\gamma_2E_n\right)\\
=&\left(\sum_{s_1=1}^{k+1}\alpha_{1,s_1}B_1^{s_1}+\sum_{t_1=1}^{k+1}\lambda_{1,t_1}B_2^{t_1}+\sum_{i\in W}\sum_{j\in M}\mu_{1,i,j}E_{i,j}\right)\\
&\left(\sum_{s_2=1}^{k+1}\alpha_{2,s_2}B_1^{s_2}+\sum_{t_2=1}^{k+1}\lambda_{2,t_2}B_2^{t_2}+\sum_{i\in W}\sum_{j\in M}\mu_{2,i,j}E_{i,j}\right)\\
=&\sum_{s_1=1}^{k+1}\sum_{s_2=1}^{k+1}\alpha_{1,s_1}\alpha_{2,s_2}B_1^{s_1+s_2}+\sum_{t_1=1}^{k+1}\sum_{t_2=1}^{k+1}\lambda_{1,t_1}\lambda_{2,t_2}B_2^{t_1+t_2}.
\end{aligned}
\end{equation}
Then, the above expression can be rewritten as
\begin{equation*}
A_1A_2=\gamma_2A_1+\gamma_1A_2-\gamma_1\gamma_2E_n+\sum_{s_1=1}^{k+1}\sum_{s_2=1}^{k+1}\alpha_{1,s_1}\alpha_{2,s_2}B_1^{s_1+s_2}+\sum_{t_1=1}^{k+1}\sum_{t_2=1}^{k+1}\lambda_{1,t_1}\lambda_{2,t_2}B_2^{t_1+t_2}.
\end{equation*}
It is straightforward to observe that $A_1A_2\in \mathcal{B}_{k,m,l}$, and $\mathcal{B}_{k,m,l}$ is closed under standard matrix addition. Therefore, $\mathcal{B}_{k,m,l}$ is a subalgebra of $\mathrm{M}_n(\mathbb{F})$.

\textbf{Step 3}:~~The following proves that $\mathcal{B}_{k,m,l}$ is a commutative subalgebra.

According to Eq. (\ref{dierbu}) in Step 2, it can be concluded that
\begin{equation*}
\left(A_1-\gamma_1E_n\right)\left(A_2-\gamma_2E_n\right)=\sum_{s_1=1}^{k+1}\sum_{s_2=1}^{k+1}\alpha_{1,s_1}\alpha_{2,s_2}B_1^{s_1+s_2}+\sum_{t_1=1}^{k+1}\sum_{t_2=1}^{k+1}\lambda_{1,t_1}\lambda_{2,t_2}B_2^{t_1+t_2}.
\end{equation*}
And
\begin{equation*}
\begin{aligned}
&\left(A_2-\gamma_2E_n\right)\left(A_1-\gamma_1E_n\right)\\
=&\left(\sum_{s_2=1}^{k+1}\alpha_{2,s_2}B_1^{s_2}+\sum_{t_2=1}^{k+1}\lambda_{2,t_2}B_2^{t_2}+\sum_{i\in N}\sum_{j\in M}\mu_{2,i,j}E_{i,j}\right)\\
&\left(\sum_{s_1=1}^{k+1}\alpha_{1,s_1}B_1^{s_1}+\sum_{t_1=1}^{k+1}\lambda_{1,t_1}B_2^{t_1}+\sum_{i\in N}\sum_{j\in M}\mu_{1,i,j}E_{i,j}\right)\\
=&\sum_{s_2=1}^{k+1}\sum_{s_1=1}^{k+1}\alpha_{2,s_2}\alpha_{1,s_1}B_1^{s_1+s_2}+\sum_{t_2=1}^{k+1}\sum_{t_1=1}^{k+1}\lambda_{2,t_2}\lambda_{1,t_1}B_2^{t_1+t_2}.
\end{aligned}
\end{equation*}
Therefore, we have
$$\left(A_1-\gamma_1E_n\right)\left(A_2-\gamma_2E_n\right)=\left(A_2-\gamma_2E_n\right)\left(A_1-\gamma_1E_n\right),$$ 
which implies $A_1A_2=A_2A_1$. Thus, the subalgebra $\mathcal{B}_{k,m,l}$ is commutative.

\textbf{Step 4}:~~Verification $\mathcal{B}_{k,m,l}$ is a maximal commutative subalgebra.

For any matrix $A=\left(a_{i,j}\right)_{n\times n}\in \mathrm{M}_n\left(\mathbb{F}\right)$, the matrix $A$ commutes with every matrix in $\mathcal{B}_{k,m,l}$, which means $A$ commutes with any generator matrices of the subalgebra $\mathcal{B}_{k,m,l}$.

\begin{enumerate}[label=\circled{\arabic*}]
\item $A$ commutes with $E_{i,j}$, where $i\in W, j\in M$. Thus
\begin{equation*}
\begin{aligned}
&AE_{i,j}=\left(\sum_{k=1}^{n}\sum_{s=1}^{n}a_{k,s}E_{k,s}\right)E_{i,j}\xlongequal{s=i}\sum_{k=1}^{n}a_{k,i}E_{k,i}E_{i,j}=\sum_{k=1}^{n}a_{k,i}E_{k,j},\\
&E_{i,j}A=E_{i,j}\left(\sum_{k=1}^{n}\sum_{s=1}^{n}a_{k,s}E_{k,s}\right)\xlongequal{k=j}\sum_{s=1}^{n}a_{j,s}E_{i,j}E_{j,s}=\sum_{s=1}^{n}a_{j,s}E_{i,s}.
\end{aligned}
\end{equation*}
Therefore, for any $i\in W, j\in M$, we have
$$\sum_{k=1}^{n}a_{k,i}E_{k,j}=\sum_{s=1}^{n}a_{j,s}E_{i,s},$$
which implies that
\begin{equation*}
a_{1,i}E_{1,j}+a_{2,i}E_{2,j}+\cdots+a_{n,i}E_{n,j}=a_{j,1}E_{i,1}+a_{j,2}E_{i,2}+\cdots+a_{j,n}E_{i,n}.
\end{equation*}
Since $i\in W$, it follows that the coefficients of $E_{m+1,j}, \cdots, E_{l-1,j}, E_{l+1,j},\cdots,E_{n,j}$ are $0$, i.e.,
\begin{equation}
\label{gs1}
a_{m+1,i}=a_{m+2,i}=\cdots=a_{l-1,i}=a_{l+1,i}=\cdots=a_{n,i}=0.
\end{equation}
Moreover, since $j\in M$, the coefficients before $E_{i,1}, \cdots, E_{i,k+m}$ and $E_{i,l}, \cdots, E_{i,l+k}$ are $0$, that is,
\begin{equation}
\label{gs2}
a_{j,1}=a_{j,2}=\cdots=a_{j,k+m}=0~\text{and}~a_{j,l}=\cdots=a_{j,l+k}=0.
\end{equation}
Therefore,
\begin{equation*}
\begin{aligned}
&a_{1,i}E_{1,j}+a_{2,i}E_{2,j}+\cdots+a_{m,i}E_{m,j}+a_{l,i}E_{l,j}\\
=&a_{j,k+m+1}E_{i,k+m+1}+\cdots+a_{j,l-1}E_{i,l-1}+a_{j,l+k+1}E_{i,l+k+1}+\cdots+a_{j,n}E_{i,n}.
\end{aligned}
\end{equation*}
By comparing the coefficients of both sides of the above-mentioned equation and combining the results of Eqs. (\ref{gs1}) and (\ref{gs2}), it is easy to obtain that for any $i\in W$ and $j\in M$, every pair $(i,j)$ satisfies
\begin{align}
\label{gs3}
&a_{i,i}=a_{j,j},\\
\begin{split}
\label{gs4}
&a_{1,i}=a_{2,i}=\cdots=a_{i-1,i}=a_{i+1,i}=\cdots=a_{m,i}=\cdots=a_{n,i}=0,\\
&a_{j,1}=a_{j,2}=\cdots=a_{j,j-1}=a_{j,j+1}=\cdots=a_{j,n}=0.
\end{split}
\end{align}

\item $A$ commutes with $B_1$. A direct calculation shows that
\begin{equation*}
\begin{aligned}
AB_1&=\left(\sum_{s=1}^{n}\sum_{t=1}^{n}a_{s,t}E_{s,t}\right)\left(E_{m,m+1}+E_{m+1,m+2}+\cdots+E_{m+k,m+k+1}\right)\\
&=\sum_{s=1}^{n}\left(a_{s,m}E_{s,m+1}+a_{s,m+1}E_{s,m+2}+\cdots+a_{s,m+k}E_{s,m+k+1}\right),\\
B_1A&=\left(E_{m,m+1}+E_{m+1,m+2}+\cdots+E_{m+k,m+k+1}\right)\left(\sum_{s=1}^{n}\sum_{t=1}^{n}a_{s,t}E_{s,t}\right)\\
&=\sum_{t=1}^{n}\left(a_{m+1,t}E_{m,t}+a_{m+2,t}E_{m+1,t}+\cdots+a_{m+k+1,t}E_{m+k,t}\right).
\end{aligned}
\end{equation*}
Since $AB_1=B_1A$, it can be observed that
\begin{equation}
\label{gs5}
\begin{aligned}
&\sum_{s=m}^{m+k}\left(a_{s,m}E_{s,m+1}+a_{s,m+1}E_{s,m+2}+\cdots+a_{s,m+k}E_{s,m+k+1}\right)\\
=&\sum_{t=m+1}^{m+k+1}\left(a_{m+1,t}E_{m,t}+a_{m+2,t}E_{m+1,t}+\cdots+a_{m+k+1,t}E_{m+k,t}\right).
\end{aligned}
\end{equation}
and the following equality holds:
\begin{equation}
\label{gs6}
\left\{
\begin{aligned}
&a_{p,m}=\cdots=a_{p,m+k}=0,~\qquad p\in \{1,2,\cdots, m-1,m+k+1,\cdots,n\},\\
&a_{m+1,q}=\cdots=a_{m+k+1,q}=0,~q\in \{1,2,\cdots, m,m+k+2,\cdots,n\}.\\
\end{aligned}
\right.
\end{equation}
For every $t\in\{m+1,\cdots,m+k+1\}$, Eqs. (\ref{gs5}) and (\ref{gs6}) jointly imply the identity:
\begin{equation}
\label{gs7}
\begin{aligned}
\sum_{s=m}^{m+k}a_{s,t-1}E_{s,t}&=a_{m+1,t}E_{m,t}+a_{m+2,t}E_{m+1,t}+\cdots+a_{m+k+1,t}E_{m+k,t}\\
&=a_{m,t-1}E_{m,t}+a_{m+1,t-1}E_{m+1,t}+\cdots+a_{m+k,t-1}E_{m+k,t}.
\end{aligned}
\end{equation}
By comparing the coefficients on both sides of the above equation, we can further derive that
$$a_{m+r,t-1}=a_{m+r+1,t},$$ 
where $r\in\{0,1,\cdots,k\}$ and $t\in\{m+1,\cdots,m+k+1\}$.
Hence, a recursive application of the preceding relation yields
\begin{align}
\label{gs8}
&a_{m,m}=a_{m+1,m+1}=\cdots=a_{m+k+1,m+k+1},\\
\begin{split}
\label{gs9}
&a_{m,m+1}=a_{m+1,m+2}=\cdots=a_{m+k,m+k+1},\\
&a_{m,m+2}=a_{m+1,m+3}=\cdots=a_{m+k-1,m+k+1},\\
&\hspace{8.65em}\vdots\hspace{8.65em}\\
&a_{m,m+k}=a_{m+1,m+k+1}.
\end{split}
\end{align}
Substituting Eqs. (\ref{gs8}) and~(\ref{gs9})~into Eq. (\ref{gs7}), we obtain
\begin{equation}
\label{gs10}
\left\{
\begin{aligned}
&a_{m+1,m}=a_{m+2,m+1}=\cdots=a_{m+k+1,m+k}=0,\\
&\hspace{8.65em}\vdots\hspace{8.65em}\\
&a_{m+k,m}=a_{m+k+1,m+1}=0.\\
\end{aligned}
\right.
\end{equation}

\item $A$ commutes with $B_2$, direct calculation shows
\begin{equation*}
\begin{aligned}
AB_2&=\left(\sum_{s=1}^{n}\sum_{t=1}^{n}a_{s,t}E_{s,t}\right)\left(E_{l,l+1}+E_{l+1,l+2}+\cdots+E_{l+k,l+k+1}\right)\\
&=\sum_{s=1}^{n}\left(a_{s,l}E_{s,l+1}+a_{s,l+1}E_{s,l+2}+\cdots+a_{s,l+k}E_{s,l+k+1}\right),\\
B_2A&=\left(E_{l,l+1}+E_{l+1,l+2}+\cdots+E_{l+k,l+k+1}\right)\left(\sum_{s=1}^{n}\sum_{t=1}^{n}a_{s,t}E_{s,t}\right)\\
&=\sum_{t=1}^{n}\left(a_{l+1,t}E_{l,t}+a_{l+2,t}E_{l+1,t}+\cdots+a_{l+k+1,t}E_{l+k,t}\right).
\end{aligned}
\end{equation*}
From $AB_2=B_2A$, it immediately follows that
\begin{equation}
\label{gs11}
\begin{aligned}
&\sum_{s=l}^{l+k}\left(a_{s,l}E_{s,l+1}+a_{s,l+1}E_{s,l+2}+\cdots+a_{s,l+k}E_{s,l+k+1}\right)\\
=&\sum_{t=l+1}^{l+k+1}\left(a_{l+1,t}E_{l,t}+a_{l+2,t}E_{l+1,t}+\cdots+a_{l+k+1,t}E_{l+k,t}\right).
\end{aligned}
\end{equation}
and the following equality holds:
\begin{equation}
\label{gs12}
\left\{
\begin{aligned}
&a_{p,l}=\cdots=a_{p,l+k}=0,~~~\qquad p\in \{1,2,\cdots, l-1,l+k+1,\cdots,n\},\\
&a_{l+1,q}=\cdots=a_{l+k+1,q}=0,~ q\in \{1,2,\cdots, l,l+k+2,\cdots,n\}.\\
\end{aligned}
\right.
\end{equation}
Substituting the above result into Eq. (\ref{gs11}), we find that for each $t\in\{l+1,\cdots,l+k+1\}$:
\begin{equation}
\label{gs13}
\begin{aligned}
\sum_{s=l}^{l+k}a_{s,t-1}E_{s,t}&=a_{l+1,t}E_{l,t}+a_{l+2,t}E_{l+1,t}+\cdots+a_{l+k+1,t}E_{l+k,t}\\
&=a_{l,t-1}E_{l,t}+a_{l+1,t-1}E_{l+1,t}+\cdots+a_{l+k,t-1}E_{l+k,t}.
\end{aligned}
\end{equation}
By comparing coefficients on both sides of the above equation, we further obtain the recurrence relation: 
$$a_{l+r,t-1}=a_{l+r+1,t},$$ 
where $r\in\{0,1,\cdots,k\}$ and $t\in\{l+1,\cdots,l+k+1\}$. We therefore establish by induction that
\begin{align}
\label{gs14}
&a_{l,l}=a_{l+1,l+1}=\cdots=a_{l+k+1,l+k+1},\\
\begin{split}
\label{gs15}
&a_{l,l+1}=a_{l+1,l+2}=\cdots=a_{l+k,l+k+1},\\
&a_{l,l+2}=a_{l+1,l+3}=\cdots=a_{l+k-1,l+k+1},\\
&\hspace{8.65em}\vdots\hspace{8.65em}\\
&a_{l,l+k}=a_{l+1,l+k+1}.
\end{split}
\end{align}
Furthermore, according to Eq. (\ref{gs12}), we deduce that
\begin{equation}
\label{gs16}
\left\{
\begin{aligned}
&a_{l+1,l}=a_{l+2,l+1}=\cdots=a_{l+k+1,l+k}=0,\\
&a_{l+2,l}=a_{l+3,l+1}=\cdots=a_{l+k+1,l+k-1}=0,\\
&\hspace{8.65em}\vdots\hspace{8.65em}\\
&a_{l+k,l}=a_{l+k+1,l+1}=0.\\
\end{aligned}
\right.
\end{equation}

To summarize the results obtained thus far, from Eqs. (\ref{gs3}), (\ref{gs8}) and (\ref{gs14}), we derive the following identity:
\begin{equation}
\label{gs17}
a_{1,1}=a_{2,2}=\cdots=a_{n,n}.
\end{equation}
By Eqs. (\ref{gs9}) and (\ref{gs15}),
\begin{equation}
\label{gs18}
\left\{
\begin{aligned}
&a_{m,m+1}=a_{m+1,m+2}=\cdots=a_{m+k,m+k+1},\\
&a_{m,m+2}=a_{m+1,m+3}=\cdots=a_{m+k-1,m+k+1},\\
&\hspace{8.65em}\vdots\hspace{8.65em}\\
&a_{m,m+k}=a_{m+1,m+k+1},\\
&a_{l,l}=a_{l+1,l+1}=\cdots=a_{l+k+1,l+k+1},\\
&a_{l,l+1}=a_{l+1,l+2}=\cdots=a_{l+k,l+k+1},\\
&a_{l,l+2}=a_{l+1,l+3}=\cdots=a_{l+k-1,l+k+1},\\
&\hspace{8.65em}\vdots\hspace{8.65em}\\
&a_{l,l+k}=a_{l+1,l+k+1}.\\
\end{aligned}
\right.
\end{equation}
Also, based on Eqs. (\ref {gs10}) and (\ref {gs16}), it can be deduced that
\begin{equation}
\label{gs19}
\left\{
\begin{aligned}
&a_{m+1,m}=a_{m+2,m+1}=\cdots=a_{m+k+1,m+k}=0,\\
&\hspace{8.65em}\vdots\hspace{8.65em}\\
&a_{m+k,m}=a_{m+k+1,m+1}=0,\\
&a_{l+1,l}=a_{l+2,l+1}=\cdots=a_{l+k+1,l+k}=0,\\
&a_{l+2,l}=a_{l+3,l+1}=\cdots=a_{l+k+1,l+k-1}=0,\\
&\hspace{8.65em}\vdots\hspace{8.65em}\\
&a_{l+k,l}=a_{l+k+1,l+1}=0.\\
\end{aligned}
\right.
\end{equation}
Finally, from Eqs. (\ref{gs4}), (\ref{gs6}) and (\ref{gs12}), we can obtain that
\begin{equation}
\label{gs20}
\left\{
\begin{aligned}
&a_{1,i}=\cdots=a_{m,i}=\cdots=a_{l-1,i}=a_{l,i}=\cdots=a_{n,i}=0,\\
&a_{j,1}=\cdots=a_{j,k+m}=\cdots=a_{j,j-1}=a_{j,j+1}=\cdots=a_{j,n}=0,\\
&a_{p_1,m}=\cdots=a_{p_1,m+k}=0,\quad\quad\quad p_1\in \{1,2,\cdots, m-1,m+k+1,\cdots,n\},\\
&a_{p_2,l}=\cdots=a_{p_2,l+k}=0,\qquad \quad \quad p_2\in \{1,2,\cdots, l-1,l+k+1,\cdots,n\},\\
&a_{m+1,q_1}=\cdots=a_{m+k+1,q_1}=0,~\quad q_1\in \{1,2,\cdots, m,m+k+2,\cdots,n\},\\
&a_{l+1,q_2}=\cdots=a_{l+k+1,q_2}=0,~\qquad q_2\in \{1,2,\cdots, l,l+k+2,\cdots,n\},\\
\end{aligned}
\right.
\end{equation}
where the parameter $i$ and $j$ range over $i\in W$ and $j\in M$, respectively.

To conclude, the general form of matrix $A$ is derived from Eqs. (\ref{gs17}) - (\ref{gs20}) as
\begin{equation*}
A=\lambda E_n+\sum_{s=1}^{k+1}\alpha_s B_1^{s}+\sum_{t=1}^{k+1}\beta_t B_2^{t}+\sum_{i\in W}\sum_{j\in M}\mu_{i,j} E_{i,j}.
\end{equation*}
Obviously, $A\in \mathcal{B}_{k,m,l}$. Therefore, $\mathcal{B}_{k,m,l}$ is a maximal commutative subalgebra of the full matrix algebra $\mathrm{M}_n(\mathbb{F})$.
\end{enumerate}
\end{proof}

\subsection{The length of the maximal commutative subalgebra $\mathcal{B}_{k,m,l}$}
\ \newline

In this section, we calculate the length of the maximal commutative subalgebra $\mathcal{B}_{k,m,l}$ and provide examples of the algebra $\mathcal{B}_{k,m,l}$ and the maximal commutative subalgebra $\mathcal{B}_{k,m}$ in Proposition \ref{mingtijida} respectively for comparison. The main results of this section are presented below.

\begin{theo}
\label{dinglijida1}
Let $m, l\in \mathbb{N}$, $k\in \mathbb{Z}^+$, $l>m+k+1$, and $l+k+1\le n$. Consider the subalgebra $\mathcal{B}_{k,m,l}\subseteq \mathrm{M}_n(\mathbb{F})$ generated by the matrices
$$\mathbb{E}_n, B_1=E_{m,m+1}+E_{m+1,m+2}+\cdots+E_{m+k,m+k+1}, $$
$$B_2=E_{l,l+1}+E_{l+1,l+2}+\cdots+E_{l+k,l+k+1}, E_{i,j},$$
where $1\le i\le m$ or $i=l$, and $m+k+1\le j\le l-1$ or $l+k+1\le j\le n$. Then $\ell\left(\mathcal{B}_{k,m,l}\right)=k+1.$
\end{theo}

\begin{proof}
According to Step 1 in the proof of Theorem \ref{dinglijida}, it is easily seen that the algebra $\mathcal{B}_{k,m,l}$ has the nilpotency index of the radical $N=k+2$. Theorem \ref{dingli2} implies that
$$\ell\left(\mathcal{B}_{k,m,l}\right)\le N-1=k+1.$$

It remains to prove that $\ell\left(\mathcal{B}_{k,m,l}\right)\ge k+1$. We consider the following generating system for $\mathcal{B}_{k,m,l}$:
\begin{equation*}
\mathcal{S}=\{B_1, B_2, E_{i,j}\vert i\in W, j\in M, (i,j)\neq (m,m+k+1), (i,j)\neq (l,l+k+1)\}.
\end{equation*}
Note that the following equation holds:
\begin{equation*}
\begin{aligned}
&B_1B_2=B_2B_1=0, \\
&B_vE_{i,j}=E_{i,j}B_v=0, ~\text{where}~v=1,2, \\
&\text{For any}~E_{i,j}, E_{p,q} \in \mathcal{S},~\text{we have that}~E_{i,j}E_{p,q}=0.   \\
\end{aligned}
\end{equation*}
By construction, computing directly gives
\begin{equation*}
\begin{aligned}
&B_1^s=\sum_{h=0}^{k-s+1}E_{m+h,m+h+s}\in L_s(S),~\text{where}~s\in\{1, 2, \cdots, k+1\}, \\
&B_2^t=\sum_{h=0}^{k-t+1}E_{l+h,l+h+t}\in L_t(S),~\text{where}~t\in\{1, 2, \cdots, k+1\}.\\
\end{aligned}
\end{equation*}
and we also have that $B_1^s\notin L_{s-1}(S)~\text{and}~B_2^t\notin L_{t-1}(S)$ for any $s, t\in \{2,\cdots,k+1\}$. Thus, $\ell\left(\mathcal{B}_{k,m,l}\right)\ge k+1$. 

Therefore, $\ell\left(\mathcal{B}_{k,m,l}\right)= k+1$.
\end{proof}

Let us consider the following explicit example:

\begin{example}
\label{li1}
When $n=8$, $m=1$, $l=5$, $k=2$, we consider the subalgebra $\mathcal{B}_{2,1,5}\subseteq \mathrm{M}_8(\mathbb{F})$ generated by the following matrices
$$\mathbb{E}_8, B_1=E_{1,2}+E_{2,3}+E_{3,4}, B_2=E_{5,6}+E_{6,7}+E_{7,8}, E_{i,j},$$
where $i=1$ or $i=5$, and $j=4$ or $j=8$. Then $\mathcal{B}_{2,1,5}$ is a maximal commutative subalgebra in $\mathrm{M}_8(\mathbb{F})$ of length $\ell\left(\mathcal{B}_{2,1,5}\right)=3$.
\end{example}

\begin{proof}
The commutativity and maximality of algebra $\mathcal{B}_{2,1,5}$ have been established in Theorem \ref{dinglijida}. Below, we provide an intuitive proof of its maximality.

By definition \ref{dingyijida}, every matrix $A=\left(a_{ij}\right)_{8\times 8}\in \mathrm{M}_8\left(\mathbb{F}\right)$ commutes with all matrices in algebra $\mathcal{B}_{2,1,5}$, that is, $A$ commutes with any generator matrices of $\mathcal{B}_{2,1,5}$. The matrix $A$ can be partitioned in the following form:
$$
	A=
	\left(
	\begin{array}{c:c}
	A_1 & A_2 \\ \hdashline
	A_3 & A_4
	\end{array}
	\right),
$$
where
\begin{equation*}
\textbf{$A_1$}=
\begin{pmatrix}
a_{11} & a_{12} & a_{13} & a_{14} \\
a_{21} & a_{22} & a_{23} & a_{24} \\
a_{31} & a_{32} & a_{33} & a_{34} \\
a_{41} & a_{42} & a_{43} & a_{44} \\
\end{pmatrix},\quad
\textbf{$A_2$}=
\begin{pmatrix}
a_{15} & a_{16} & a_{17} & a_{18} \\
a_{25} & a_{26} & a_{27} & a_{28} \\
a_{35} & a_{36} & a_{37} & a_{38} \\
a_{45} & a_{46} & a_{47} & a_{48} \\
\end{pmatrix},
\end{equation*}
\begin{equation*}
\textbf{$A_3$}=
\begin{pmatrix}
a_{51} & a_{52} & a_{53} & a_{54} \\
a_{61} & a_{62} & a_{63} & a_{64} \\
a_{71} & a_{72} & a_{73} & a_{74} \\
a_{81} & a_{82} & a_{83} & a_{84} \\
\end{pmatrix},\quad
\textbf{$A_4$}=
\begin{pmatrix}
a_{55} & a_{56} & a_{57} & a_{58} \\
a_{65} & a_{26} & a_{67} & a_{68} \\
a_{75} & a_{36} & a_{77} & a_{78} \\
a_{85} & a_{46} & a_{87} & a_{88} \\
\end{pmatrix}.
\end{equation*}

\begin{itemize}
\item If $A$ commutes with $B_1=E_{1,2}+E_{2,3}+E_{3,4}$, regarding the matrix $B_1$ as
$$
    B=
	\left(
	\begin{array}{c:c}
	B_1& \text{\large 0}_4 \\ \hdashline
	\text{\large 0}_4& \text{\large 0}_4
	\end{array}
	\right),
$$
then $AB=BA$, it implies that 
$$
	\left(
	\begin{array}{c:c}
	A_1& A_2 \\ \hdashline
	A_3& A_4
	\end{array}
	\right)
\left(
	\begin{array}{c:c}
	B_1& \text{\large 0}_4 \\ \hdashline
	\text{\large 0}_4& \text{\large 0}_4
	\end{array}
	\right)
=
\left(
	\begin{array}{c:c}
	B_1& \text{\large 0}_4 \\ \hdashline
	\text{\large 0}_4& \text{\large 0}_4
	\end{array}
	\right)
\left(
	\begin{array}{c:c}
	A_1& A_2 \\ \hdashline
	A_3& A_4
	\end{array}
	\right).
$$
Hence,
$$
	\left(
	\begin{array}{c:c}
	A_1B_1& \text{\large 0}_4 \\ \hdashline
	A_3B_1& \text{\large 0}_4
	\end{array}
	\right)
=
\left(
	\begin{array}{c:c}
	B_1A_1& B_1A_2 \\ \hdashline
	\text{\large 0}_4& \text{\large 0}_4
	\end{array}
	\right).
$$
Observing the above matrix equality, we obtain $A_1B_1=B_1A_1$, $B_1A_2=A_3B_1=\text{\large 0}_4$. Consequently,

\begin{enumerate}[label=\circled{\arabic*}]
\item\label{xuhao1} Since $A_1B_1=B_1A_1$,
\begin{equation*}
\textbf{$A_1B_1$}=
\begin{pmatrix}
a_{11} & a_{12} & a_{13} & a_{14} \\
a_{21} & a_{22} & a_{23} & a_{24} \\
a_{31} & a_{32} & a_{33} & a_{34} \\
a_{41} & a_{42} & a_{43} & a_{44} \\
\end{pmatrix}
\begin{pmatrix}
0 & 1 & 0 & 0 \\
0 & 0 & 1 & 0 \\
0 & 0 & 0 & 1 \\
0 & 0 & 0 & 0 \\
\end{pmatrix}
=
\begin{pmatrix}
0&a_{11} & a_{12} & a_{13}  \\
0&a_{21} & a_{22} & a_{23}  \\
0&a_{31} & a_{32} & a_{33}  \\
0&a_{41} & a_{42} & a_{43}  \\
\end{pmatrix},
\end{equation*}
\begin{equation*}
\textbf{$B_1A_1$}=
\begin{pmatrix}
0 & 1 & 0 & 0 \\
0 & 0 & 1 & 0 \\
0 & 0 & 0 & 1 \\
0 & 0 & 0 & 0 \\
\end{pmatrix}
\begin{pmatrix}
a_{11} & a_{12} & a_{13} & a_{14} \\
a_{21} & a_{22} & a_{23} & a_{24} \\
a_{31} & a_{32} & a_{33} & a_{34} \\
a_{41} & a_{42} & a_{43} & a_{44} \\
\end{pmatrix}
=
\begin{pmatrix}
a_{21} & a_{22} & a_{23} & a_{24} \\
a_{31} & a_{32} & a_{33} & a_{34} \\
a_{41} & a_{42} & a_{43} & a_{44} \\
0      &0       &0       &0\\
\end{pmatrix}.
\end{equation*}
By comparing the aforementioned matrix identities, we derive
\begin{equation}
\label{23}
\left\{
\begin{aligned}
&a_{21}=a_{31}=a_{32}=a_{41}=a_{42}=a_{43}=0,\\
&a_{11}=a_{22}=a_{33}=a_{44},\\
&a_{12}=a_{23}=a_{34},\\
&a_{13}=a_{24}.\\
\end{aligned}
\right.
\end{equation}
Therefore, $A_1$ can be expressed as an upper triangular matrix
\begin{equation*}
\textbf{$A_1$}=
\begin{pmatrix}
a & b & c & d \\
0 & a & b & c \\
0 & 0 & a & b \\
0 & 0 & 0 & a \\
\end{pmatrix}.
\end{equation*}

\item\label{xuhao2} Since $B_1A_2=A_3B_1=\text{\large 0}_4$,
\begin{equation*}
\textbf{$B_1A_2$}=
\begin{pmatrix}
0 & 1 & 0 & 0 \\
0 & 0 & 1 & 0 \\
0 & 0 & 0 & 1 \\
0 & 0 & 0 & 0 \\
\end{pmatrix}
\begin{pmatrix}
a_{15} & a_{16} & a_{17} & a_{18} \\
a_{25} & a_{26} & a_{27} & a_{28} \\
a_{35} & a_{36} & a_{37} & a_{38} \\
a_{45} & a_{46} & a_{47} & a_{48} \\
\end{pmatrix}
=
\begin{pmatrix}
a_{25} & a_{26} & a_{27} & a_{28} \\
a_{35} & a_{36} & a_{37} & a_{38} \\
a_{45} & a_{46} & a_{47} & a_{48} \\
0&0&0&0
\end{pmatrix},
\end{equation*}
\begin{equation*}
\textbf{$A_3B_1$}=
\begin{pmatrix}
a_{51} & a_{52} & a_{53} & a_{54} \\
a_{61} & a_{62} & a_{63} & a_{64} \\
a_{71} & a_{72} & a_{73} & a_{74} \\
a_{81} & a_{82} & a_{83} & a_{84} \\
\end{pmatrix}
\begin{pmatrix}
0 & 1 & 0 & 0 \\
0 & 0 & 1 & 0 \\
0 & 0 & 0 & 1 \\
0 & 0 & 0 & 0 \\
\end{pmatrix}
=
\begin{pmatrix}
0&a_{51} & a_{52} & a_{53} \\
0&a_{61} & a_{62} & a_{63} \\
0&a_{71} & a_{72} & a_{73} \\
0&a_{81} & a_{82} & a_{83} \\
\end{pmatrix}.
\end{equation*}
Upon comparison, it becomes evident that
\begin{equation}
\label{gs24}
\left\{
\begin{aligned}
&a_{25}=a_{26}=a_{27}=a_{28}=0,\\
&a_{35}=a_{36}=a_{37}=a_{38}=0,\\
&a_{45}=a_{46}=a_{47}=a_{48}=0.\\
\end{aligned}
\right.
\end{equation}
and
\begin{equation}
\label{gs25}
\left\{
\begin{aligned}
&a_{51}=a_{52}=a_{53}=0,\\
&a_{61}=a_{62}=a_{63}=0,\\
&a_{71}=a_{72}=a_{73}=0,\\
&a_{81}=a_{82}=a_{83}=0.\\
\end{aligned}
\right.
\end{equation}
\end{enumerate}

\item Using an approach similar to the above, we derive $A_4B_2=B_2A_4$, $A_2B_2=B_2A_3=\text{\large 0}_4$ from $AB_2=B_2A$. Consequently, the block $A_4$ can be expressed as
\begin{equation*}
\textbf{$A_4$}=
\begin{pmatrix}
\alpha & \beta & \gamma & \eta \\
0 & \alpha & \beta & \gamma \\
0 & 0 & \alpha & \beta \\
0 & 0 & 0 & \alpha \\
\end{pmatrix}.
\end{equation*}
Meanwhile, from $A_2B_2=B_2A_3=\text{\large 0}_4$, it follows that
\begin{equation}
\label{gs26}
\left\{
\begin{aligned}
&a_{15}=a_{16}=a_{17}=0,\\
&a_{25}=a_{26}=a_{27}=0,\\
&a_{35}=a_{36}=a_{37}=0,\\
&a_{45}=a_{46}=a_{47}=0.\\
\end{aligned}
\right.
\end{equation}
and
\begin{equation}
\label{gs27}
\left\{
\begin{aligned}
&a_{61}=a_{62}=a_{63}=a_{64}=0,\\
&a_{71}=a_{72}=a_{73}=a_{74}=0,\\
&a_{81}=a_{82}=a_{83}=a_{84}=0.\\
\end{aligned}
\right.
\end{equation}

\item If $A$ commutes with $E_{1,4}$, $E_{1,8}$, $E_{5,4}$, $E_{5,8}$. The analysis is restricted to the case where $A$ commutes with $E_{1,4}$; other cases admit similar derivations.

Since
\begin{equation*}
\begin{aligned}
&AE_{1,4}=\left(\sum_{k=1}^{8}\sum_{s=1}^{8}a_{ks}E_{k,s}\right)E_{1,4}=\sum_{k=1}^{8}a_{k1}E_{k,4},\\
&E_{1,4}A=E_{1,4}\left(\sum_{k=1}^{8}\sum_{s=1}^{8}a_{ks}E_{k,s}\right)=\sum_{s=1}^{8}a_{4s}E_{1,s},
\end{aligned}
\end{equation*}
thus,
\begin{equation}
\label{gs28}
\left\{
\begin{aligned}
&a_{11}=a_{44}, \\
&a_{21}=a_{32}=\cdots=a_{81}=0,\\
&a_{41}=a_{42}=a_{43}=a_{45}=\cdots=a_{84}=0.
\end{aligned}
\right.
\end{equation}
Similarly, we can obtain
\begin{equation}
\label{gs29}
\left\{
\begin{aligned}
&a_{11}=a_{88}=a_{55}=a_{44}, \\
&a_{21}=a_{32}=\cdots=a_{81}=0,\\
&a_{81}=\cdots=a_{87}=0,\\
&a_{15}=\cdots=a_{45}=a_{65}=\cdots=a_{85}=0.\\
\end{aligned}
\right.
\end{equation}
\end{itemize}

In summary, according to Eqs. (\ref{gs24}) - (\ref{gs27}), the forms of blocks $A_2$ and $A_3$ can be respectively derived as
\begin{equation*}
\textbf{$A_2$}=
\begin{pmatrix}
0 & 0 & 0 & a_{18} \\
0 & 0 & 0 & 0 \\
0 & 0 & 0 & 0 \\
0 & 0 & 0 & 0 \\
\end{pmatrix},
\textbf{$A_3$}=
\begin{pmatrix}
0 & 0 & 0 & a_{54} \\
0 & 0 & 0 & 0 \\
0 & 0 & 0 & 0 \\
0 & 0 & 0 & 0 \\
\end{pmatrix}.
\end{equation*}
Moreover, from Eq. (\ref{gs29}), we conclude that any matrix $A$ commuting with all generator matrices of the algebra $\mathcal{B}_{2,1,5}$ must of the following form:
$$
A=
\left(
\begin{array}{cccc:cccc}
a & b & c & d      & 0   & 0      & 0      & \mu     \\
0 & a & b & c      & 0   & 0      & 0      & 0          \\
0 & 0 & a & b      & 0   & 0      & 0      & 0          \\
0 & 0 & 0 & a      & 0   & 0      & 0      & 0          \\ \hdashline
0 & 0 & 0 &\lambda & a   & \beta  & \gamma & \eta       \\
0 & 0 & 0 & 0      & 0   & a      & \beta  & \gamma     \\
0 & 0 & 0 & 0      & 0   & 0      & a      & \beta      \\
0 & 0 & 0 & 0      & 0   & 0      & 0      & a          \\
\end{array}
\right),
$$
that is, $A=aE_8+b B_1+c B_1^2+d B_1^3+\beta B_2+\gamma B_2^2+\eta B_2^3+\mu E_{1,8}+\lambda E_{5,4}\in\mathcal{B}_{2,1,5}$. Therefore, $\mathcal{B}_{2,1,5}$ is a maximal commutative subalgebra.

Next, we compute the length of $\mathcal{B}_{2,1,5}$. From the previous results, we can see that its nilpotency index is $N=4$. According to Theorem \ref{dingli2}, we conclude that
$$\ell\left(\mathcal{B}_{2,1,5}\right)\le N-1=3.$$

To prove the lower bound $\ell\left(\mathcal{B}_{2,1,5}\right)\ge k+1$, we consider the following generating system $\mathcal{S}$ for $\mathcal{B}_{2,1,5}$:
\begin{equation*}
\mathcal{S}=\{B_1, B_2, E_{1,8}, E_{5,4}\}.
\end{equation*}
Note that $B_1^3=E_{1,4}$ and $B_2^3=E_{5,8}$. Thus, the matrix set $\mathcal{S}$ is a generating system for $\mathcal{B}_{2,1,5}$. Meanwhile, we observe that the following equations
\begin{equation*}
\begin{aligned}
&B_1B_2=B_2B_1=0, \\
&B_vE_{i,j}=E_{i,j}B_v=0, ~\text{where}~v=1,2, \\
&E_{1,8}E_{5,4}=E_{5,4}E_{1,8}=0   \\
\end{aligned}
\end{equation*}
hold. Moreover, from the special properties of matrices $B_1$ and $B_2$, we obtain
\begin{equation*}
\begin{aligned}
&B_1^s=\sum_{h=0}^{3-s}E_{1+h,1+h+s}\in L_s(S),~\text{where}~s\in\{1, 2, 3\}, \\
&B_2^t=\sum_{h=0}^{3-t}E_{5+h,5+h+t}\in L_t(S),~\text{where}~t\in\{1, 2, 3\}.\\
\end{aligned}
\end{equation*}
and we have that $B_1^s\notin L_{s-1}(S)~\text{and}~B_2^t\notin L_{t-1}(S)$ for any $s, t\in \{2, 3\}$, that is, $B_1^s, B_2^t$ can not be represented as a linear combination of words of smaller length. Thus, $\ell\left(\mathcal{B}_{2,1,5}\right)\ge 3$. Therefore, $\ell\left(\mathcal{B}_{2,1,5}\right)= 3$.

\end{proof}

Below we provide an illustration based on Proposition \ref{mingtijida} and compare it with Example \ref{li1}.

\begin{example}
When $m=1$, $n=8$, $k=2$, we consider the subalgebra $\mathcal{B}_{2,1}\subseteq\mathrm{M}_8(\mathbb{F})$ generated by the following matrices
$$\mathbb{E}_8, B=E_{1,2}+E_{2,3}+E_{3,4}, E_{1,j},$$
where $4\le j\le 8$, and $j\in\mathbb{Z}$. Then $\mathcal{B}_{2,1}\subseteq \mathrm{M}_8(\mathbb{F})$ is a maximal commutative subalgebra of length $\ell\left(\mathcal{B}_{2,1}\right)=3$.
\end{example}

\begin{proof}
By direct computation, we have that $BE_{1,j}=E_{1,j}B=0$ and $E_{1,j}E_{1,p}=0$ for any $E_{1,j}$, $E_{1,p}$. By construction, we also have that
$$B^s=\displaystyle{\sum_{h=0}^{3-s}}E_{1+h,1+h+s},~\text{where}~s\in\{1, 2, 3\}.$$

Therefore, since all generator matrices commute, the commutativity of $\mathcal{B}_{2,1}$ follows immediately. We now prove the maximality of algebra $\mathcal{B}_{2,1}$. 

From the definition of a maximal commutative subalgebra~\ref{dingyijida}, every matrix $A=\left(a_{i,j}\right)_{8\times 8}$ in $\mathrm{M}_8\left(\mathbb{F}\right)$ commutes with all matrices in $\mathcal{B}_{2,1}$, it implies that $A$ commutes with any generator matrices of $\mathcal{B}_{2,1}$. The matrix $A$ admits the following block decomposition:
$$
	A=
	\left(
	\begin{array}{c:c}
	A_{11}& A_{12} \\ \hdashline
	A_{21}& A_{22} \\
	\end{array}
	\right),
$$
where the blocks $A_{11}$ and $A_{22}$ are $4\times 4$ matrices. The generator matrices $B, E_{1,j}$ can be expressed as
$$
	B_{1}=
	\left(
	\begin{array}{c:c}
	B& \text{\large 0}_4 \\ \hdashline
	\text{\large 0}_4& \text{\large 0}_4 \\
	\end{array}
	\right),
\quad
    E_{1,j}=
	\left(
	\begin{array}{c:c}
	\text{\large 0}_4& C \\ \hdashline
	\text{\large 0}_4& \text{\large 0}_4 \\
	\end{array}
	\right).
$$
\begin{enumerate}[label=\circled{\arabic*}]
\item If matrices $A$ and $B$ commute, Eqs. (\ref{23}) -~(\ref{gs25}) follow from computations parallel to those in Example \ref{li1}, \ref{xuhao1}, and \ref{xuhao2}. Consequently, block $A_{11}$ takes the following form:
\begin{equation*}
\textbf{$A_{11}$}=
\begin{pmatrix}
\alpha & \beta & \gamma & \lambda \\
0 & \alpha & \beta & \gamma \\
0 & 0 & \alpha & \beta \\
0 & 0 & 0 & \alpha \\
\end{pmatrix}.
\end{equation*}
\item Since $A$ commutes with $E_{1,j}$, where $4\le j\le 8$ and $j\in\mathbb{Z}$. According to $AE_{1,j}=E_{1,j}A$, we obtain
\begin{equation*}
\begin{aligned}
&AE_{1,j}=\left(\sum_{k=1}^{8}\sum_{s=1}^{8}a_{k,s}E_{k,s}\right)E_{1,j}=\sum_{k=1}^{8}a_{k,1}E_{k,j},\\
&E_{1,j}A=E_{1,j}\left(\sum_{k=1}^{8}\sum_{s=1}^{8}a_{k,s}E_{k,s}\right)=\sum_{s=1}^{8}a_{j,s}E_{1,s}.
\end{aligned}
\end{equation*}
Consequently, by comparing the coefficients on both sides of the above equation, we obtain the following system of equations
\begin{equation}
\label{gs30}
\left\{
\begin{aligned}
&a_{1,1}=a_{j,j}, \\
&a_{2,1}=a_{3,1}=\cdots=a_{8,1}=0,\\
&a_{j,1}=a_{j,2}=\cdots=a_{j,j-1}=a_{j,j+1}=\cdots=0.
\end{aligned}
\right.
\end{equation}
\end{enumerate}
Consequently, according to Eqs. (\ref{23}) - (\ref{gs25}) and~(\ref{gs30}), we find that the matrix $A$ commuting with all generator matrices of $\mathcal{B}_{2,1}$ takes the form:
$$
A=
\left(
\begin{array}{cccc:cccc}
\alpha & \beta   & \gamma  & \lambda & a_{1,5}& a_{1,6}& a_{1,7}& a_{1,8}     \\
0      & \alpha  & \beta   & \gamma  & 0      & 0      & 0      & 0          \\
0      & 0       & \alpha  & \beta   & 0      & 0      & 0      & 0          \\
0      & 0       & 0       & \alpha  & 0      & 0      & 0      & 0          \\ \hdashline
0      & 0       & 0       & 0       & \alpha & 0      & 0      & 0          \\
0      & 0       & 0       & 0       & 0      & \alpha & 0      & 0          \\
0      & 0       & 0       & 0       & 0      & 0      & \alpha & 0          \\
0      & 0       & 0       & 0       & 0      & 0      & 0      & \alpha     \\
\end{array}
\right),
$$
that is, $A=\alpha E_8+\beta B+\gamma B^2+\lambda B^3+\sum_{j=5}^{8}a_{1,j}E_{1,j} \in\mathcal{B}_{2,1}$. Thus, $\mathcal{B}_{2,1}$ is a maximal commutative subalgebra.

Moreover, the length of $\mathcal{B}_{2,1}$ can be computed following a procedure analogous to Example \ref{li1}. Therefore, $\ell\left(\mathcal{B}_{2,1}\right)= 3.$
\end{proof}

\hspace*{\fill}

A comparison of these two examples reveals fundamentally distinct constructions of the maximal commutative subalgebras. However, the methods for computing the lengths of algebras $\mathcal{B}_{m,k,l}$ and $\mathcal{B}_{m,k}$ are analogous. The main idea is to obtain a specific generator matrix through linear combinations of generator matrices, thereby identifying a generating system $\mathcal{S}$ that establishes the upper bound for the algebra's length. Moreover, we find that the lengths of these two classes of algebras depend solely on the parameter $k$, which to some extent reflects the complexity of the algebraic structures.

\hspace*{\fill}

The author declares no conflict of interest.

\end{document}